\documentclass{article}
\usepackage{graphicx} % Required for inserting images
\usepackage[utf8]{inputenc}
\usepackage{amsmath}
\usepackage{amssymb}
\usepackage{textcomp}
\usepackage[english]{babel}
\usepackage{amsthm}
\usepackage{hyperref}
\usepackage{tikz}
\usetikzlibrary{arrows,decorations.pathmorphing,backgrounds,positioning,fit,matrix}
\theoremstyle{definition}
\newtheorem{definition}{Definition}[section]
\newtheorem{remark}{Remark}
\newtheorem{theorem}{Theorem}[section]
\newtheorem{lemma}[theorem]{Lemma}

\newtheorem{claim}[theorem]{Claim}
\newtheorem{outline}[theorem]{Outline}
\newtheorem{fact}[theorem]{Fact}

\newtheorem{notation}[theorem]{Notation}

\newtheorem*{acknowledgements}{Acknowledgements}
\title{Towards Erd\H{o}s-Hajnal property for
dp-minimal 
graphs}
\author{Yayi Fu }
\date{}
\begin{document}
\maketitle
\begin{abstract}
We introduce the notion of strongly $\binom{k}{2}$-free graphs, 
which contain dp-minimal graphs.
We show that under some sparsity assumption, 
given a rainbow $\binom{k}{2}$-free blockade we can find a rainbow $\binom{k-1}{2}$-free blockade. 
This might serve as an intermediate
step towards Erd\H os-Hajnal property for dp-minimal graphs.
\end{abstract} \hspace{10pt}
\section{Introduction}
\indent

\emph{Erd\H{o}s-Hajnal conjecture} \cite{erdos1989ramsey} says for any graph $H$ there is $\epsilon>0$ such that if a graph $G$ does not contain any induced subgraph isomorphic to $H$ then $G$ has a clique or an anti-clique of size $\geq |G|^\epsilon$.
More generally, we say a family of finite graphs has the \emph{Erd\H{o}s-Hajnal property} if there is $\epsilon>0$ such that for any graph $G$ in the family, $G$ has a clique or an anti-clique of size $\geq|G|^\epsilon$. 
Malliaris and Shelah proved in \cite{malliaris2014regularity} that the family of stable graphs has the Erd\H{o}s-Hajnal property. 
Chernikov and Starchenko gave another proof for stable graphs in \cite{chernikov2018note} and in \cite{chernikov2018regularity} they proved that the family of distal graphs has the strong Erd\H{o}s-Hajnal property. 
In general, we are interested in whether the family of finite VC-dimension (i.e. NIP \cite{simon2015guide}) graphs, which contains both stable graphs and distal graphs, has the Erd\H{o}s-Hajnal property. 
Motivation for studying this problem was given in \cite{fox2019erdHos}, which also gave a lower bound $e^{(\log n)^{1-o(1)}}$ for largest clique or anti-clique in a graph with bounded VC dimension. 
In this paper, we consider graphs in dp-minimal theories, a special case of NIP graphs.
\\
\indent 
This paper is inspired by \cite{chudnovsky2021erdos}. 
Techniques used also come from
\cite{chudnovsky2021erdos}.
\\
\indent
We will prove the following lemma: 
\begin{lemma}
    (main lemma)
Given $k\in\mathbb{N},
d\in\mathbb{R}$ with 
$k\geq 2$, $d\geq 2$, 
there exists $\tau_0=\tau_0(k,d)$, 
$L_0=L_0(k,d)$ satisfying the following:
\\
\indent
Let $\tau<\tau_0$, $G$ a
strongly $\binom{k}{2}$-free
$\tau$-critical graph,
and $\mathcal{A}=(A_i:1\leq i\leq t)\subseteq G$
an equicardinal blockade of width $\frac{|G|}{t^d}$ with $\frac{|G|}{t^{2d}}\leq W_G$,
of length 
$L_0 \leq t\leq 
2|G|^{\frac{1}{d}}$ such that for all $a\in A$,
$|E(a,A)|<
\frac{|G|}{t^{d}}$.
\\
\indent
Then there exist
$b\in A$,
an $(t', \frac{|G|}{t^{2d+2}})$-comb $((a_j,A_j'):1\leq j\leq t')$ in
\\
$(E_b\cap A,
\neg E_b\cap A)$ such that $\mathcal{A}'=(A_j':1\leq j\leq t')$ is an equicardinal minor of $\mathcal{A}$ with width $\geq
\frac{|G|}{t^{2d+2}}$, 
length $\geq
t^{\frac{1}{8}}$.
\end{lemma}
Section \ref{prelim} gives the definition of dp-minimality and 
combinatorial notions we need for the proof.
Section \ref{proof} proves the lemma. 
\begin{acknowledgements}
The author is grateful to her advisor Sergei Starchenko for helpful suggestions.
Also many thanks to Istv\'an Tomon and 
Alex Scott for pointing out mistakes and giving comments.
\end{acknowledgements}
\section{Preliminaries}\label{prelim}
\indent

As usual, a \textit{graph} is a pair $G=(V,E)$ where $V$ is a finite set, $E$ is a binary symmetric 
anti-reflexive relation on $V$.
\begin{definition}\cite[Definition~2.1]{https://doi.org/10.48550/arxiv.0910.3189}
Fix a structure $\mathcal{M}$. An \emph{ICT pattern} in $\mathcal{M}$ consists of a pair of formulas $\varphi(x,\Bar{y})$ and $\psi(x,\Bar{y})$ and sequences $\{\Bar{a}_i :i\in\omega\}$ and $\{\Bar{b}_i :i\in\omega\}$ from $\mathcal{M}$ so that, for all $i, j \in\omega$, the following is consistent:
$\varphi(x,\Bar{a}_i)\wedge\psi(x,\Bar{b}_j)\underset{l\neq i}{\bigwedge}\neg\varphi(x,\Bar{a}_l)\underset{k\neq j}{\bigwedge}\neg\psi(x,\Bar{b}_k)$.
\end{definition}
\begin{remark}    
By compactness, for a pair of formulas $\varphi$, $\psi$, if for all $n\in \omega$, there exist $\{\Bar{a}_i :i\in n\}$, $\{\Bar{b}_i :i\in n\}$ such that for all $i,j\in n$, $\varphi(x,\Bar{a}_i)\wedge\psi(x,\Bar{b}_j)\underset{l\neq i}{\bigwedge}\neg\varphi(x,\Bar{a}_l)\underset{k\neq j}{\bigwedge}\neg\psi(x,\Bar{b}_k)$ is consistent, then there is an ICT pattern in $\mathcal{M}$ if $\mathcal{M}$ is sufficiently saturated.
\end{remark}
\begin{definition}\cite[Definition~2.3]{https://doi.org/10.48550/arxiv.0910.3189}
    A theory $T$ is said to be \emph{dp-minimal} if there is no ICT pattern in any model $\mathcal{M}\models T$.
\end{definition}
\begin{definition}
The family $\mathcal{F}$ of cographs is the smallest family of graphs
satisfying the following:
\begin{enumerate}
    \item if $G$ is a graph with a single vertex, 
    then $G\in \mathcal{F}$;
    \item if $G$ is in $\mathcal{F}$,
    then $\overline{G}$ is in 
    $\mathcal{F}$;
    \item if $G_1=(V_1,E_1)$
    and $G_2=(V_2,E_2)$
    are in $\mathcal{F}$, 
    then the graph $G=(V,E)$
with $V=V_1\dot\cup V_2$
and $E=E_1\cup E_2$ is in $\mathcal{F}$.
 \end{enumerate}
\indent

Any such $G\in\mathcal{F}$
is called a \emph{cograph}.
\end{definition}
As a corollary,
we also have if $G_1=(V_1,E_1)$
    and $G_2=(V_2,E_2)$
    are cographs, 
    then the graph $G=(V,E)$
with $V=V_1\dot\cup V_2$
and $E=E_1\cup E_2
\cup\{(x,y):x\in V_1, y\in V_2\}$ is a cograph, 
because this is the complement of the disjoint union 
of $\overline{G_1}$ and $\overline{G_2}$.
\\
\indent
Also, it is well-known that every cograph $G$ 
contains a homogeneous set of size $\geq |G|^{\frac{1}{2}}$.
\begin{definition}
    We say that a graph $G$ is \emph{$\tau$-critical} if the largest size of a cograph in $G$ is $<|G|^{\tau}$, and for every induced subgraph $G'$ of $G$ with $G'\neq G$, there is a cograph in $G'$ of size $\geq |G'|^{\tau}$ .
\end{definition}
\begin{definition}\cite[Section~5]{chudnovsky2021erdos}
    Let $G$ be a graph.  A \emph{blockade} $\mathcal{B}$ in $G$ is a sequence $(B_1,..., B_t)$ of pairwise disjoint
subsets of $V (G)$ called \emph{blocks}. We denote $B_1\cup ...\cup B_t$ by $B$. The length of a blockade is the number of blocks,
and its width is the minimum cardinality of a block.\\
\indent
A \emph{pure pair} in $G$ is a pair $A$, $B$ of disjoint subsets of $V(G)$ such that $A$ is either
complete or anticomplete to $B$.
A blockade $\mathcal{B} = (B_1, ... , B_t)$ in $G$ is \emph{pure} if $(B_i
, B_j )$ is a pure pair for all $i, j$ with $1\leq i < j \leq t$.
\\
\indent
For a pure blockade $\mathcal{B}=(B_1,...,B_t)$,
let $P$ be the graph with vertex set $\{1,..., t\}$, in which $i, j$ are adjacent if $B_i$
is complete to $B_j$. We say $P$ is the \emph{pattern} of the pure blockade $\mathcal{B}$.
\end{definition}
\begin{definition}\cite[Section~2]{chudnovsky2020pure}
    If $\mathcal{B} = (B_i: i \in I)$ is a blockade, let $I'\subseteq I$; then $(B_i: i \in I')$ is a blockade, of smaller length but of at least the same width, and we call it a \emph{sub-blockade} of $\mathcal{B}$. Second, for each $i \in I$ let $B_i'\subseteq B_i$ be nonempty; then the sequence $(B_i': i \in I)$ is a blockade, of the same length but possibly of smaller width, and we call it a \emph{contraction of $\mathcal{B}$}. A contraction of a sub-blockade (or equivalently, a sub-blockade of a contraction) we call a \emph{minor} of $\mathcal{B}$.
\\
\indent
\cite[Section~3]{chudnovsky2020pure}
    Say a blockade is \emph{equicardinal} if all its blocks have the same cardinality. 
\end{definition}
\begin{definition}
    Let $G=(V,E)$ be a graph,
          $X\subseteq V$, 
          $k\in\mathbb{N}$ and 
          $a_1,...,a_k\in V$.
          We say 
          \emph{$a_1,...,a_k$ has
          $\binom{k}{2}$-property 
          over $X$} if there exists 
          $\{b_{ij}\}_{
          1\leq i<j\leq s}\subseteq X$ such that 
          for each pair $i\neq 
          j,1\leq i<j\leq s$, 
          $E(b_{ij},a_i)\wedge 
E(b_{ij},a_j)\wedge\underset{m\neq i,m\neq j}{\bigwedge}\neg E(b_{ij},a_m)$.
    \\
    \indent
    Given a blockade $\mathcal{B}=(B_i:1\leq i\leq t)$ in $G$, and $a_1,...,a_s\in V$,
    we say $(a_1,...a_s)$ is a \emph{$\mathcal{B}$-rainbow tuple} if for any $i\neq j$,
    there exist $i'\neq j'$ such that $a_i\in B_{i'}$ and $a_j\in B_{j'}$.
    \\
    \indent
    For $k\geq 2$, we say $G$ is \emph{$\binom{k}{2}$-free} if there is no $k$-tuple $(a_1,...,a_k)$ of distinct vertices in $G$ with 
    $\binom{k}{2}$-property over $G$.
     $G$ is \emph{strongly
          $\binom{k}{2}$-free}
          if both $G$ and $\overline{G}$
          are $\binom{k}{2}$-free.
    \\
    \indent
    Given a blockade $\mathcal{B}=(B_i:1\leq i\leq t)$ in $G$, we say $\mathcal{B}$ is \emph{rainbow 
    $\binom{k}{2}$-free} if there is no $\mathcal{B}$-rainbow tuple $(a_1,...,a_k)$ with $\binom{k}{2}$-property over $B$.
\end{definition}
\begin{definition}
\cite[Section~2]{chudnovsky2021erdos}
    Let $G$ be a graph, and let
    $t, k \geq 0$ where $t$ is an integer. We say $
    ((a_i, B_i) : 1 \leq i \leq t)$ is a \emph{$(t, k)$-comb} in $G$ if:
    \begin{itemize}
        \item 
 $a_1,..., a_t \in V (G)$ are distinct, and $B_1,..., B_t$ are pairwise disjoint subsets of 
 $V(G)\setminus\{a_1,..., a_t\}$;
\item 
for $1 \leq i \leq t$, $a_i$
is adjacent to every vertex in $B_i$;
\item 
for $i, j \in \{1,..., t\}$ with $i\neq j$, $a_i$ has no neighbour in $B_j$ ; and
\item 
$B_1,..., B_t$ all have cardinality at least $k$.
\end{itemize}
If $C, D \subseteq V (G)$ are disjoint and $a_1,..., a_t \in C$, and $B_1,..., B_t \subseteq D$, we call this a \emph{$(t, k)$-comb in
$(C, D)$}. 
\end{definition}
The following fact originated from \cite{tomon2023string} and was stated in \cite{PACH202121}.
It was also adopted in \cite{chudnovsky2021erdos}. 
We use here the version in \cite[2.1]{chudnovsky2021erdos}.
\begin{fact}\label{comblem}
    Let $G$ be a graph with a bipartition $(A,B)$, such that every vertex in $B$ has a neighbour in $A$; and let $\Gamma,\Delta,d > 0$ with $d < 1$, such that every vertex in $A$ has at most $\Delta$ neighbours in $B$. Then either:
    \begin{itemize}
        \item 
    for some integer $t \geq 1$, there is a $(t, \Gamma t ^{-1/d})$-comb in $(A, B)$; or 
    \item $|B| \leq \frac{3^{d+1}}{3/2-(3/2)^d} \Gamma^{d}\Delta^{1-d}$.
    \end{itemize}
\end{fact}
\begin{notation}
    Let $G$ be a graph and $X\subseteq V(G)$.
    Then $G[X]$ denotes the induced subgraph of $G$ on the subset $X$ and 
    $\overline{G}[X]$ denotes the complement of $G[X]$.
\end{notation}
\begin{fact}\cite[4.3]{chudnovsky2021erdos} \label{rodl} For every graph $H$ and all $\epsilon > 0$, there exists $\delta > 0$ such that for every $H$-free graph $G$, there
exists $X \subseteq V (G)$ with $|X|\geq\delta|G|$, such that one of $G[X]$, $\overline{G}[X]$ has maximum degree at most $\epsilon\delta|G|$.
\end{fact}
\begin{claim}
\label{basecase}
(Variant of \cite[6.7]
{chudnovsky2021erdos}): For all 
$s\geq 1$, let $D_s=2^{s-
1}d^{2s-1}$, $d=4$.
Let 
$\mathcal{B}=(B_i:i=1,...,D_s)$ 
be a rainbow $\binom{2}{2}$-free blockade in $G$ of length 
$D_s$ and width $W$. Then $G$ admits a 
pure blockade $\mathcal{A}$ 
with a cograph pattern, of 
length $2^s$ and width $\geq 
W/D_s$.
\end{claim}
The proof is similar to that of 
\cite[6.7]{chudnovsky2021erdos}.
\begin{proof}
For $s=1$, $D_s=4$. Let 
$X=B_1\cup B_2$, $Y=B_3\cup 
B_4$. Let $A=\{x\in X: \forall 
y\in B_3$ $ \neg E(x,y)\}$, 
$A'=\{x\in X: \forall y\in B_4 
$ $\neg E(x,y)\}$. 
By assumption, 
$A\cup A'=X$.
Hence $|A|\geq 
\frac{1}{2}|X|$ or 
$|A'|\geq\frac{1}{2}|X|$. 
In 
either case, the conclusion 
holds.\\
For $s+1$, 
$D_{s+1}=2^sd^{2s+1}$.
Let $L=B_1\cup...\cup 
B_{\frac{D_{s+1}}
{4}}$, 
$R=B_{\frac{D_{s+1}}
{4}+1}\cup...\cup 
B_{\frac{D_{s+1}}
{2}}$,
$L'=\{x\in (B_1\cup...\cup
B_{D_{s+1}})\setminus (L\cup 
R): 
\forall y\in L$ $\neg E(x,y)\}$, 
$R'=
\{x\in (B_1\cup...\cup 
B_{D_{s+1}})\setminus (L\cup 
R): 
\forall y\in R$ $\neg E(x,y)\}$.
Then $L'\cup R'=(B_1\cup...\cup 
B_{D_{s+1}})\setminus (L\cup R)$
and $|L'|\geq\frac{1}{2}|
(B_1\cup...\cup 
B_{D_{s+1}})\setminus 
(L\cup R)|$ or $|R'|\geq\frac{1}
{2}|
(B_1\cup...\cup 
B_{D_{s+1}})\setminus 
(L\cup R)|$.
In either case,
we have pure pairs $(A,B)$ such 
that 
$|A|,|B|\geq \frac{WD_{s+1}}
{d^2}$. 
The rest is the same as in 
\cite[6.7]
{chudnovsky2021erdos}
\end{proof}
\begin{remark}
\label{baserem}
    Given $\mathcal{B}=
    (B_i:i=1,...,t)$ be a 
rainbow $\binom{2}{2}$-free blockade in 
$G$ 
of length $t\geq 4$ and width 
$W$,
let $s$ be a positive integer 
such 
that $D_s\leq t<D_{s+1}$,
where $D_s$ is as defined in claim \ref{basecase}.
Since for $s\geq 1$,
$2^s\geq 
D_{s+1}^{\frac{1}{10}}
=2^{\frac{s}{2}+\frac{1}{5}}
\geq t^{\frac{1}{10}}$,
and $\frac{W}{D_s}\geq
\frac{W}{t}$,
by claim \ref{basecase}, 
$G$ has 
    a pure blockade 
    $\mathcal{A}$ with a 
    cograph pattern,
    of length $\geq t^{\frac{1}
    {10}}$
    and width $\geq \frac{W}
    {t}$.
\end{remark}
\section{Proof}\label{proof}
\begin{outline}
    We start from a simple observation: 
    claim \ref{keyob} says that given a 
    rainbow  $\binom{k}{2}$-free blockade
    $\mathcal{A}
    =
    (A_i:1\leq i\leq t)$, 
    we can find a rainbow
    $\binom{k-1}{2}$-free minor if there exist $a\in V(A)$ and 
    $((a_j,B_j):
    1\leq j\leq t')$ a comb in 
    $(E_a, \neg E_a)$ such that 
    $\mathcal{B}=
    (B_j: 1\leq j\leq t')$ is
    a minor of $\mathcal{A}$ 
    and for all $j$, 
    $B_j\cap A_{i_0}=\emptyset$,
    where $a\in A_{i_0}$.
    Lemma \ref{keylem}
says we can construct such a comb in a $\tau$-critical
graph $G$ in a given \textquotedblleft
sparse" blockade
$\mathcal{A}$. 
\end{outline}
\begin{claim}
\label{keyob}
    Let $\mathcal{A}
    =
    (A_i:1\leq i\leq t)$ be a rainbow 
    $\binom{k}{2}$-free blockade. 
    If $a\in V(A)$ and 
    $((a_j,B_j):
    1\leq j\leq t')$ is a comb in 
    $(E_a\cap A, 
    \neg E_a\cap A)$ such that 
    $\mathcal{B}=
    (B_j: 1\leq j\leq t')$ is
    a minor of $\mathcal{A}$ 
    and for all $j$, 
    $B_j\cap A_{i_0}=\emptyset$, 
    where $a\in A_{i_0}$,
    then 
    $\mathcal{B}$ is
    rainbow $\binom{k-1}{2}$-free.
\end{claim}
\begin{proof}
      Let $\mathcal{A}$, 
    $\mathcal{B}$ be as in the claim.
    Suppose $(b_1,...,b_{k-1})$ is a 
    $\mathcal{B}$-rainbow tuple with 
    $\binom{k-1}{2}$-property over $B$, 
    witnessed by 
    $\{c_{lm}\}_{l\neq m,
    1\leq l,m \leq k-1}$.
    Then $(a,b_1,...,b_{k-1})$
    has $\binom{k}{2}$-property 
    over $A$, 
    witnessed by 
    $\{c_{lm}\}_{l\neq m,
    1\leq l,m \leq k-1}
    \cup 
    \{a_j\}_{1\leq j\leq t'}$,
    a contradiction.
\end{proof}
\begin{notation}
    Let $G$ be a graph. 
    Let $\mathcal{F}_G
    =
\{((b_i,B_i):1\leq i\leq s)\subseteq G:\\
\exists a\in V[G]$ $((b_i,B_i):1\leq i\leq s)$ is a $(s,\frac{|G|}{s^2})$-comb in $(E_a,\neg E_a)$ and $(B_i:1\leq i\leq s)$ is an equicardinal blockade $\}$. 
Then let \emph{$W_G$} denote minimal width of equicardinal blockades $(B_i:1\leq i\leq s)$ satisfying that there exist $(b_i:1\leq i\leq s)$ such that $((b_i,B_i):1\leq i\leq s)$ is in $\mathcal{F}_G
\cup 
\mathcal{F}_{\overline{G}}$, if $\mathcal{F}_G
\cup
\mathcal{F}_{\overline{G}}
\neq \emptyset$; let $W_G=|G|$,
if $\mathcal{F}_G
\cup 
\mathcal{F}_{\overline{G}}=\emptyset$.  
\end{notation}
\begin{lemma}
\label{keylem}
(main lemma)
Given $k\in\mathbb{N},
d\in\mathbb{R}$ with 
$k\geq 2$, $d\geq 2$, 
there exists $\tau_0=\tau_0(k,d)$, 
$L_0=L_0(k,d)$ satisfying the following:
\\
\indent
Let $\tau<\tau_0$, $G$ a
strongly $\binom{k}{2}$-free
$\tau$-critical graph,
and $\mathcal{A}=(A_i:1\leq i\leq t)\subseteq G$
an equicardinal blockade of width $\frac{|G|}{t^d}$ with $\frac{|G|}{t^{2d}}\leq W_G$,
of length 
$L_0 \leq t\leq 
2|G|^{\frac{1}{d}}$ such that for all $a\in A$,
$|E(a,A)|<
\frac{|G|}{t^{d}}$.
\\
\indent
Then there exist
$b\in A$,
an $(t', \frac{|G|}{t^{2d+2}})$-comb $((a_j,A_j'):1\leq j\leq t')$ in
\\
$(E_b\cap A,
\neg E_b\cap A)$ such that $\mathcal{A}'=(A_j':1\leq j\leq t')$ is an equicardinal minor of $\mathcal{A}$ with width $\geq
\frac{|G|}{t^{2d+2}}$, 
length $\geq
t^{\frac{1}{8}}$.
\end{lemma}
\begin{outline}
    We will construct a sequence
    $(a_u,\Delta_u,R_u)$:\\
    \indent
   We define $R_0=A$, $a_0\in R_0$ such that
    $E(a_0,R_0)$ is maximum,
    $\Delta_0=|E(a_0,R_0)|$.
    \\
    \indent
    Suppose  $(a_u,\Delta_u,R_u)$ is 
    constructed.
   If $\Delta_u\geq
   \frac{|G|}{t^{2d}}$,
    we use the idea of
    fact \ref{comblem} to construct the
    comb we want. 
    If a comb in lemma
    \ref{keylem} 
    exists, then we end construction;
    if it does not exist,
    then the set
    $\{y\in \neg E(a_u,R_u):
    \exists x\in E(a_u,R_u) $ $E(x,y)\}$ has size 
    bounded above by $\frac{|G|}{t^d}
    \cdot t^{\frac{1}{4}}$.
    We take $R_{u+1}$ to be 
    $\{y\in \neg E(a_u,R_u):
    \neg\exists x\in E(a_u,R_u)$ $E(x,y)\}$,
    $a_{u+1}\in R_{u+1}$ such that
    $E(a_{u+1},R_{u+1})$ is maximum,
    $\Delta_{u+1}=|E(a_{u+1},R_{u+1})|$.
    If $\Delta_u<\frac{|G|}{t^{2d}}$,
    by the choice of $W_G$ and
    by \cite[2.1]{chudnovsky2021erdos},
    the set
    $\{y\in \neg E(a_u,R_u):
    \exists x\in E(a_u,R_u) $ $E(x,y)\}$ has size 
    bounded above by $\frac{3^{\frac{1}{2}+1}}{\frac{3}{2} -(\frac{3}{2})^{\frac{1}{2}}}
    \cdot
    \frac{|G|}{t^{d}}
$.
 We take $R_{u+1}$ to be 
    $\{y\in \neg E(a_u,R_u):
    \neg\exists x\in E(a_u,R_u)$ $E(x,y)\}$,
    $a_{u+1}\in R_{u+1}$ such that
    $E(a_{u+1},R_{u+1})$ is maximum,
    $\Delta_{u+1}=|E(a_{u+1},R_{u+1})|$.
    \\
    \indent
      Because in the construction of the $R_u$'s, 
      we removed a set of 
      \textquotedblleft
      small" size
      at each step,
    $R_{\lceil
    t^{\frac{1}{8}}
    \rceil
    }$ has size at least 
    $\frac{|G|}{t^d}
    \cdot t^{\frac{1}{2}}$.
    We then discuss two cases depending on the size of
    $\Delta_{\lceil
    t^{
    \frac{1}{8}
    }\rceil}$.
    \\
    \indent
    We observe that for all $u<u'$,
    for all $x\in E(a_u;R_u)$, 
    $y\in E(a_{u'};R_{u'})$, $\neg E(x,y)$.
    If $\Delta_{\lceil
    t^{
    \frac{1}{8}
    }\rceil}
    \geq
    \frac{|G|}{t^{2d}}$, 
    then $\bigl(E(a_u;R_u):
1\leq u\leq \lceil
t^{\frac{1}{8}}\rceil\bigr)$ is a 
pure blockade of cograph pattern
and since $G$ is $\tau$-critical,
we have in $G$ a cograph of size
$\geq
\frac{|G|^{\tau}}{t^{2d\tau}}
\cdot
t^{\frac{1}{8}}
\geq
|G|^{\tau}$
contradicting the choice of $\tau$;
if $\Delta_{\lceil
    t^{
    \frac{1}{8}
    }\rceil}
    <
    \frac{|G|}{t^{2d}}$,
    then we follow the proof of 
    \cite[3.1]{chudnovsky2021erdos}
    to get a contradiction. 
\end{outline}
\begin{proof}
Let $k,d$ be given. Let $K=\underset{\alpha=1}{\overset{\infty}{\sum}} (\frac{2}{3})^{\alpha}$. Let $L_0$ be the smallest positive integer such that
for all $L\geq L_0$,
\begin{equation*}
    L^\frac{1}{4}\geq 
(3+\frac{9}{2}K)L^{\frac{1}{8}}
+
3
+
 \frac{3^{\frac{1}{2}+1}}{\frac{3}{2} -(\frac{3}{2})^{\frac{1}{2}}},
 \end{equation*}
 \begin{equation*}
   L-
2
L^{\frac{1}{8}}
(1
+
2^d
+
L^{\frac{1}{4}}
)
\geq
 L^{\frac{1}{2}}.
 \end{equation*}
 \\
 \indent
 Let $\tau_0>0$ be small such that for all $\tau<\tau_0$,
 \begin{equation*}
     \tau-\frac{1}{2d}
     <
     -(d+1)\tau,
 \end{equation*}
 \begin{equation*}
     L_0^{-d-\frac{1}{2}+2d\tau}
     +
     \frac{3^{\frac{1}{2}+1}}{\frac{3}{2} -(\frac{3}{2})^{\frac{1}{2}}}
    \cdot
     L_0^{-\frac{1}{2}+
     2d\tau}
     +
     2^
     {-\frac{1}{2}}
       <1,
 \end{equation*}
 and
 \begin{equation*}
     L_0^{\frac{1}{8}-2d\tau}
     >1.
 \end{equation*}
 \\
\indent
Fix any $\tau<\tau_0$. Let $G$ be $\tau$-critical.
\\
\indent
Let $\mathcal{A}=(A_i:1\leq i\leq t)\subseteq G$ be an equicardinal blockade with width $\frac{|G|}{t^d}$ such that $\frac{|G|}{t^{2d}}\leq W_G$, length $L_0
\leq t\leq
2|G|^{\frac{1}{d}}$ and for all $a\in A$,
$|E(a,A)|<\frac{|G|}{t^d}$. Suppose the comb in the statement does not exist.
\\
\indent
Construct $a_u,\Delta_u,R_u$ such that 
\begin{enumerate}
    \item $\Delta_u$ is the maximal degree in $G[R_u]$
    \item $a_u\in R_u$ such that $|E(a_u,R_u)|=\Delta_u$
\end{enumerate}
as follows:\\
\indent
Let $R_0=A$. Let $\Delta_0$ 
be the maximum of 
$\{|E(x,A)|:x\in A\}$. Let 
$a_0\in R_0$ such that 
$|E(a_0,R_0)|=\Delta_0$.  
\\
\indent
Suppose $\Delta_u,R_u,a_u$ 
are constructed. We construct 
$\Delta_{u+1}, 
R_{u+1},a_{u+1}$:
\\
\indent 
Let $C=E(a_u,R_u)$.
Let $D=\neg E(a_u,R_u)$. 
\\
     \indent
     We repeat 
the construction of 
combs in \cite[2.1]
{chudnovsky2021erdos}.\\ 
Construct 
$k_s\in\mathbb{N},
a^s_{1},...,a^s_{k_s}\in C,
T^s_1,...,T^s_{k_s}\subseteq D,
C_s\subseteq D$ as 
follows:
\\
\indent
$\underline{s=1}$: Choose $a^1_1,...,a^1_{k_1}\in C$ with $k_1$ max such that for all $i$ there exist $\geq\frac{2}{3}\Delta_u$ vertices in $D$ that are adjacent to $a^1_i$ and non adjacent to $a^1_1,...,a^1_{i-1}$. Let $T^1_i=E(a^1_i,D)\setminus(E(a^1_1,D)\cup...\cup E(a^1_{i-1},D))$, $C_1=\underset{i}{\bigcup}$ $E(a^1_i,D)$. Then $C_1=\underset{i}{\bigcup}$ $T^1_i$.
\\
\indent
$\underline{s+1}$: Suppose $k_1,...,k_s\in\mathbb{N},
a^1_{1},...,a^1_{k_1},...,a^s_{1},...,a^s_{k_s}\in C,
T^1_1,...,T^1_{k_1},...,T^s_1,...,T^s_{k_s}\subseteq D,
C_1,...,C_s\subseteq D$ are defined. Then by maximality of $k_s$, every vertex in $C$ has $<(\frac{2}{3})^s\Delta_u$ neighbours in $F$, where $F=D\setminus (C_1\cup...\cup C_s)$. Choose $a^{s+1}_1,...,a^{s+1}_{k_{s+1}}\in C$ with $k_{s+1}$ max such that for all $i$ there exist $\geq(\frac{2}{3})^{s+1}\Delta_u$ vertices in $F$ that are adjacent to $a^{s+1}_i$ and non adjacent to $a^{s+1}_1,...,a^{s+1}_{i-1}$. Let $T^{s+1}_i=E(a^{s+1}_i,F)\setminus(E(a^{s+1}_1,F)\cup...\cup E(a^{s+1}_{i-1},F))$ and $C_{s+1}=\underset{i}{\bigcup}$ $E(a^{s+1}_i,F)$. Then $C_{s+1}=\underset{i}{\bigcup}$ $T^{s+1}_i$.\\
\indent
Observe that
\begin{enumerate}
    \item 
for all $s$, all $i< i'$ and all $y\in T^{s}_{i'}$, $\neg E(a^{s}_{i},y)$. 
\item
for all $s<s'$, all $i,i'$ and all $y\in T^{s'}_{i'}$, $\neg E(a^s_i,y)$.
\item 
for all $s,i$, 
$(\frac{2}{3})^{s}\Delta_u\leq|T^s_i|\leq(\frac{2}{3})^{s-1}\Delta_u$.

 \end{enumerate}
\underline{Case 1: $\Delta_u\geq
\frac{|G|}{t^{2d}}$} \\
\indent
Let $l$ be the largest such that $(\frac{2}{3})^{l-1}\Delta_u
 \geq\frac{|G|}{t^{2d}}$.
 \\
 \indent
Construct $I_l,...,I_1$ as follows:
\\
$\underline{\text{Construct } I_l}$:
Let $I^{k_l}_l=\{a^{l}_{k_{l}}\}$. Choose $A_i$ such that $|A_i\cap T^{l}_{k_{l}}|\geq
\frac{2}{3}
\cdot
\frac{|G|}{t^{2d+1}}$. Such $A_i$ exists, since we assumed $(\frac{2}{3})^{l}\Delta_u\geq
\frac{2}{3}\cdot
\frac{|G|}{t^{2d}}$. Denote this chosen $i$ by $i^{l}_{k_{l}}$ and let $P^{l}_{k_{l}}=A_{i^{l}_{k_{l}}}\cap T^{l}_{k_{l}}$.\\
\indent
Suppose $I^m_l,...,I^{k_l}_l$ are constructed. For $a^{l}_{m-1}$, if $<\frac{1}{2}|T^{l}_{m-1}|$ many vertices of $T^{l}_{m-1}$ is adjacent to an element in $I^m_l$,
and there exists $A_i$ such that $A_i\cap P^{l}_{m'}=\emptyset$ for all $a^{l}_{m'}$ in $I^m_l$ (i.e. $i\neq i^{l}_{m'}$ for all $a^{l}_{m'}$ in $I^m_l$)
and 
\begin{equation*}
|A_i\cap T^{l}_{m-1} \cap 
\underset{a^{l}_{m'}\in I^m_l}{\bigcap}\neg E(a^{l}_{m'},D)|\geq 
\frac{|G|}{t^{2d+2}},
\end{equation*}
 then let $I^{m-1}_l=\{a^{l}_{m-1}\}\cup I^m_l$ and choose such $i$ to be $i^{l}_{m-1}$. Let $P^{l}_{m-1}=A_{i^{l}_{m-1}}\cap T^{l}_{m-1} \cap 
\underset{a^{l}_{m'}\in I^m_l}{\bigcap}\neg E(a^{l}_{m'},D)$. Otherwise, let $I^{m-1}_l=I^m_l$.\\
\indent
Take $I_l=I^1_l$.\\
\\
\indent
 Suppose $I_l,...,I_n$ are constructed.\\
$\underline{\text{Construct } I_{n-1}}$:
Let $I^{k_{n-1}+1}_{n-1}
=\emptyset$.
Suppose $I^{m+1}_{n-1},...,I^{k_{n-1}+1}_{n-1}$
are constructed.
\\
\indent
For $a^{n-1}_{m}$,
if $<\frac{1}{2}|T^{n-1}_{m}|$
many vertices of $T^{n-1}_{m}$ 
is adjacent to an element in 
$I^{m+1}_{n-1}
\cup I_{n}\cup ...\cup I_{l}$ 
and there exists $A_i$ such 
that $A_i\cap 
P^{\alpha}_{m'}=\emptyset$ 
for all $a^{\alpha}_{m'}\in 
I^{m+1}_{n-1}
\cup I_{n}\cup...\cup I_{l}$
and 
\begin{equation*} 
|A_i\cap T^{n-1}_{m} \cap 
\underset{a^{\alpha}_{m'}\in
I^{m+1}_{n-1}\cup I_n\cup...\cup I_{l}}{\bigcap}\neg E(a^{l}_{m'},D)|\geq 
\frac{|G|}{t^{2d+2}},
\end{equation*}
then set $I^m_{n-1}=\{a^{n-1}_{m}\}\cup I^{m+1}_{n-1}$ and choose such $i$ to be $i^{n-1}_{m}$. Let $P^{n-1}_{m}=A_{i^{n-1}_{m}}\cap T^{n-1}_{m} \cap 
\underset{a^{\alpha}_{m'}\in
I^{m+1}_{n-1}
\cup I_{n}\cup...\cup I_{l}}{\bigcap}\neg E(a^{l}_{m'},D)$. Otherwise, let $I^m_{n-1}=I^{m+1}_{n-1}$. Take $I_{n-1}=I^1_{n-1}$.
\\\\
\indent
Then $((a^{\alpha}_{m},P^{\alpha}_{m}):a^{\alpha}_{m}\in I_1\cup...\cup I_l)$ is a $(E_{a_u}\cap A,
\neg E_{a_u}\cap A)$-comb such that for $(\alpha,m)\neq (\alpha',m')$, $i^{\alpha}_{m}\neq i^{\alpha'}_{m'}$ with width $\geq
\frac{|G|}{t^{2d+2}}$. If $|I_1\cup...\cup I_l|
\geq
t^{\frac{1}{8}}$, then it's a comb satisfying the statement of the lemma, a contradiction. Hence $|I_1\cup...\cup I_l|
<
t^{\frac{1}{8}}$ and we bound the size of $C_1\cup...\cup C_l=\underset{\alpha\in\{1,...,l\},m\in\{1,...,k_{\alpha}\}}{\bigcup}T^{\alpha}_m$ as follows:\\
\\
\indent
 For each 
 $a^{\alpha}_{m}\notin 
 I_1\cup...\cup I_l$ with 
 $\alpha\in\{1,...,l\}$, 
 $\geq\frac{1}
 {2}\cdot|T^{\alpha}_{m}|$ 
 many vertices of 
 $T^{\alpha}_{m}$ are in 
 $\underset{x\in 
 (I_{\alpha}\cap
 \{a^{\alpha}_{m+1},...,
 a^{\alpha}_{k_{\alpha}}\})
 \cup I_{\alpha+1}...\cup I_l}
 {\bigcup}E(x,R_u)$ or\\
$<\frac{1}
{2}\cdot|T^{\alpha}_{m}|$ 
many vertices of 
$T^{\alpha}_{m}$ are in 
$\underset{x\in 
(I_{\alpha}\cap
\{a^{\alpha}_{m+1},...,
a^{\alpha}_{k_{\alpha}}\})\cup
I_{\alpha+1}...\cup I_l}
{\bigcup}E(x,R_u)$,
but for each $i\in 
[t]\setminus
\{i^{\alpha'}_{m'}:
a^{\alpha'}_{m'}\in 
I_{\alpha+1}\cup...\cup I_l$ 
or $a^{\alpha'}_{m'}\in 
I_{\alpha}$ with $m'>m \}$, 
$|A_i\cap T^{\alpha}_{m}
\cap 
\underset{a^{\alpha'}_{m'}\in
(I_{\alpha}\cap
\{a^{\alpha}_{m+1},...,
a^{\alpha}_{k_{\alpha}}\})
\cup 
I_{\alpha+1}\cup...\cup I_{l}}
{\bigcap}\neg 
E(a^{\alpha'}_{m'},D)|
<
\frac{|G|}{t^{2d+2}}$.
\\\\
\indent
Let $S_1=
\{a^{\alpha}_{m}\notin 
I_1\cup...\cup I_l :$ 
$\alpha\in\{1,...,l\}$, 
$\geq\frac{1}
{2}\cdot|T^{\alpha}_{m}|$ 
many vertices of 
$T^{\alpha}_{m}$ are in 
$\underset{x\in 
(I_{\alpha}\cap
\{a^{\alpha}_{m+1},...,
a^{\alpha}_{k_{\alpha}}\})\cup
I_{\alpha+1}...\cup I_l}
{\bigcup}E(x,R_u)$\}. 
By construction of the 
$T^\alpha_m$'s,
for each fixed pair $\alpha, m$, 
$|T^\alpha_m|\geq
(\frac{2}
{3})^{\alpha}\Delta_u$,
and 
\begin{equation*}
    T^\alpha_m\cap
\underset{x\in 
(I_{\alpha}\cap
\{a^{\alpha}_{m+1},...,
a^{\alpha}_{k_{\alpha}}\})\cup
I_{\alpha+1}...\cup I_l}
{\bigcup}E(x,R_u)\subseteq
\end{equation*}
\begin{equation*}
    T^\alpha_m\cap
\underset{x\in 
I_{\alpha}\cup 
I_{\alpha+1}\cup...\cup I_l}
{\bigcup}E(x,R_u)
=
\end{equation*}
\begin{equation*}
\underset{x\in 
I_{\alpha}\cup 
I_{\alpha+1}\cup...\cup I_l}
{\bigcup}E(x,T^\alpha_m)
\subseteq
\end{equation*}
\begin{equation*}
    \underset{x\in 
I_{\alpha}\cup I_{\alpha+1}\cup...\cup I_l}{\bigcup}E(x,C^\alpha).
\end{equation*}
By construction, for each $x\in I_{\alpha}\cup I_{\alpha+1}\cup ...\cup I_l$,
$|E(x, C_\alpha)|<(\frac{2}{3})^{\alpha-1}\Delta_u$. Thus, for any fixed $\alpha\in
\{1,...,l\}$, we have 
\begin{equation*}  
(|\{x\in S_1: x=a^\alpha_m \text{ for some }m\in\{1,...,k_\alpha\}\}|)\cdot\frac{1}{2}\cdot(\frac{2}{3})^{\alpha}\Delta_u\leq
    (|I_\alpha|+...+|I_l|)\cdot (\frac{2}{3})^{\alpha-1}\Delta_u.
    \end{equation*}
Hence there exist at most $3\cdot(|I_{\alpha}|+...+|I_{l}|)$ many $a^{\alpha}_m$'s in $S_1$ for each fixed $\alpha$ and 
\begin{equation*}
    |\underset{a^{\alpha}_{m}\in S_1}{\bigcup}T^{\alpha}_m|
\leq 
\underset{\alpha=1}{\overset{l}{\sum}}
    3\cdot
    (\frac{2}{3})^{\alpha-1}\Delta_u\cdot(|I_\alpha|+...+|I_l|).
    \end{equation*}
\indent
Let $S_2=\Bigl\{a^{\alpha}_{m}\notin I_1\cup...\cup I_l :$ $\alpha\in\{1,...,l\}$, $<\frac{1}{2}\cdot|T^{\alpha}_{m}|$ many vertices \\\\
of
$T^{\alpha}_{m}$ are in $\underset{x\in (I_{\alpha}\cap\{a^{\alpha}_{m+1},...,a^{\alpha}_{k_{\alpha}}\})\cup I_{\alpha+1}\cup...\cup I_l}{\bigcup}E(x,R_u)$, 
\\
but for each $i\in 
[t]\setminus \{i^{\alpha'}_{m'}:a^{\alpha'}_{m'}\in (I_{\alpha}\cap\{a^{\alpha}_{m+1},...,a^{\alpha}_{k_{\alpha}}\})\cup I_{\alpha+1}\cup...\cup I_l\}$,
\begin{equation*}
|A_i\cap T^{\alpha}_{m}\cap 
\underset{a^{\alpha'}_{m'}\in
(I_{\alpha}\cap\{a^{\alpha}_{m+1},...,a^{\alpha}_{k_{\alpha}}\})
\cup I_{j(\alpha+1)}\cup...\cup I_{l}}{\bigcap}\neg E(a^{\alpha'}_{m'},R_u)|<\frac{|G|}{t^{2d+2}}\Bigr\}.
\end{equation*}
Define $S^\alpha_m:=
 \{i^{\alpha'}_{m'}:a^{\alpha'}_{m'}\in (I_{\alpha}\cap\{a^{\alpha}_{m+1},...,a^{\alpha}_{k_{\alpha}}\})\cup I_{\alpha+1}\cup...\cup I_l\}$.\\
Then
\begin{equation*}
|\underset{a^{\alpha}_{m}\in S_2}{\bigcup}
(T^{\alpha}_m
\cap 
\underset{a^{\alpha'}_{m'}\in
(I_{\alpha}\cap\{a^{\alpha}_{m+1},...,a^{\alpha}_{k_{\alpha}}\})
\cup I_{\alpha+1}\cup...\cup I_{l}}{\bigcap}\neg E(a^{\alpha'}_{m'},R_u))|=
\end{equation*}
\begin{equation*}
    |\underset{a^{\alpha}_{m}\in S_2}{\bigcup}
\Bigl(\underset{i\in [t]\setminus S^\alpha_m}{\bigcup} (A_i\cap T^{\alpha}_m
\cap 
\underset{a^{\alpha'}_{m'}\in
(I_{\alpha}\cap\{a^{\alpha}_{m+1},...,a^{\alpha}_{k_{\alpha}}\})
\cup I_{\alpha+1}\cup...\cup I_{l}}{\bigcap}\neg E(a^{\alpha'}_{m'},R_u))
\cup
\end{equation*}
\begin{equation*}
\underset{i\in S^\alpha_m}{\bigcup} (A_i\cap T^{\alpha}_m
\cap 
\underset{a^{\alpha'}_{m'}\in
(I_{\alpha}\cap
\{a^{\alpha}_{m+1},...,
a^{\alpha}_{k_{\alpha}}\})
\cup I_{\alpha+1}\cup...\cup 
I_{l}}{\bigcap}\neg 
E(a^{\alpha'}_{m'},R_u))
\Bigr)|
\end{equation*}
Since there are $t$ many 
$A_i$'s and because 
$T^\alpha_m\geq
(\frac{2}
{3})^\alpha\Delta_u\geq 
\frac{2}{3}
\cdot
\frac{|G|}{t^{2d}}$ for all 
$\alpha\in\{1,...,l\}$, 
there are at most $
\frac{|G|
\cdot
t}{t^{d}
\cdot 
\frac{2}{3}\cdot
\frac{|G|}{t^{2d}}
}
=\frac{3}{2}
\cdot
t^{d+1}$ 
many $T^{\alpha}_m$ with 
$\alpha\in\{1,...,l\}$, 
we have, 
\begin{equation*}
|\underset{a^{\alpha}_{m}
\in S_2}{\bigcup}
\underset{i\in [t]\setminus 
S^\alpha_m}{\bigcup} (A_i\cap 
T^{\alpha}_m
\cap 
\underset{a^{\alpha'}_{m'}\in
(I_{\alpha}\cap\{a^{\alpha}_{m+1},...,a^{\alpha}_{k_{\alpha}}\})
\cup I_{\alpha+1}\cup...\cup I_{l}}{\bigcap}\neg E(a^{\alpha'}_{m'},R_u))|
\leq
\frac{|G|}{t^{2d
+2}}
\cdot
t\cdot
\frac{3}{2}
\cdot
t^{d+1}
\end{equation*}
Since for each $i\in[t]$, $|A_i|=\frac{|G|}{t^d}$, we have,
\begin{equation*}
|\underset{a^{\alpha}_{m}\in S_2}{\bigcup}
\underset{i\in S^\alpha_m}{\bigcup} (A_i\cap T^{\alpha}_m
\cap 
\underset{a^{\alpha'}_{m'}\in
(I_{\alpha}\cap\{a^{\alpha}_{m+1},...,a^{\alpha}_{k_{\alpha}}\})
\cup I_{\alpha+1}\cup...\cup I_{l}}{\bigcap}\neg E(a^{\alpha'}_{m'},R_u))
|
\leq
\end{equation*}
\begin{equation*}
|\underset{i\in\{i^{\alpha'}_{m'}:
a^{\alpha'}_{m'}\in
I_1\cup...\cup I_l
 \}   }
    {\bigcup}A_i|\leq
\end{equation*}
\begin{equation*}
\frac{|G|}{t^{d}}\cdot
(|I_{1}|+...+|I_l|).
\end{equation*}
Hence, 
\begin{equation*}
|\underset{a^{\alpha}_{m}\in S_2}{\bigcup}
(T^{\alpha}_m
\cap 
\underset{a^{\alpha'}_{m'}\in
(I_{\alpha}\cap\{a^{\alpha}_{m+1},...,a^{\alpha}_{k_{\alpha}}\})
\cup I_{\alpha+1}\cup...\cup I_{l}}{\bigcap}\neg E(a^{\alpha'}_{m'},R_u))|
\leq
\frac{|G|}{t^{2d
+2}}
\cdot
t\cdot
\frac{3}{2}
\cdot
t^{d+1}
+
\frac{|G|}{t^{d}}\cdot
(|I_{1}|+...+|I_l|). 
\end{equation*}
Since for each $a^{\alpha}_{m}\in S_2$, $<\frac{1}{2}\cdot|T^{\alpha}_{m}|$ many vertices of $T^{\alpha}_{m}$ are in $\underset{x\in (I_{\alpha}\cap\{a^{\alpha}_{m+1},...,a^{\alpha}_{k_{\alpha}}\})\cup I_{1}\cup...\cup I_l}{\bigcup}E(x,R_u)$, we have 
\begin{equation*}
|\underset{a^{\alpha}_{m}\in S_2}{\bigcup}
T^{\alpha}_m|
\leq
2\cdot
(
\frac{|G|}{t^{2d+2}}
\cdot
t\cdot
\frac{3
}{2}
\cdot
t^{d+1} +
\frac{|G|}{t^{d}}\cdot
(|I_{1}|+...+|I_l|)).
\end{equation*}
Hence, we have
\begin{equation*}
    |\underset{s\leq l}{\bigcup}\text{ }C_s|
    =
    |\underset{\alpha\in\{1,...,l\},m\in\{1,...,k_{\alpha}\}}{\bigcup}T^{\alpha}_m|
    =
    \end{equation*}
    \begin{equation*}
    |\underset{a^\alpha_m\in I_1\cup...\cup I_l}{\bigcup}
    T^\alpha_m
    \cup
    \underset{a^{\alpha}_{m}\in S_1}{\bigcup}
T^{\alpha}_m
\cup
\underset{a^{\alpha}_{m}\in S_2}{\bigcup}T^{\alpha}_m|
\leq
\end{equation*}
\begin{equation*}
   \frac{|G|}{t^{d}}\cdot(|I_1|+...+|I_l|)
   +
   \underset{\alpha=1}{\overset{l}{\sum}}
    3\cdot
    (\frac{2}{3})^{\alpha-1}\Delta_u\cdot(|I_\alpha|+...+|I_l|)
    +
    2\cdot
(
\frac{|G|}{t^{2d+2}}
\cdot
t\cdot
\frac{3}{2}
\cdot
t^{d+1} +
\frac{|G|}{t^{d}}\cdot
(|I_{1}|+...+|I_l|))
\leq
\end{equation*}
\begin{equation*}
   \frac{|G|}{t^{d}}\cdot t^{\frac{1}{8}}
   +
   3\cdot \frac{|G|}{t^{d}}\cdot \frac{3}{2}\cdot
   \underset{\alpha=1}
   {\overset{l}
{\sum}}\underset{\alpha'
=1}
{\overset{\alpha}{\sum}} (\frac{2}
{3})^{\alpha'}|I_\alpha|
   +
   2\cdot
(
\frac{|G|}{t^{d}}
\cdot
\frac{3}{2}
+
\frac{|G|}{t^{d}}\cdot
t^{\frac{1}{8}})
\leq
\end{equation*}
\begin{equation*}
     \frac{|G|}{t^{d}}\cdot t^{\frac{1}{8}}
   +
   3\cdot \frac{|G|}{ t^{d}}\cdot \frac{3}{2}\cdot
   \underset{\alpha=1}{\overset{l}{\sum}} (\frac{2}{3})^{\alpha}
   \cdot(|I_1|+...+|I_l|)
   +
   2\cdot
(
\frac{|G|}{t^{d}}
\cdot
\frac{3}{2}
+
\frac{|G|}{t^{d}}\cdot
t^{\frac{1}{8}})\leq
\end{equation*}
\begin{equation*}
    \frac{|G|}{t^{d}}\cdot 
    t^{\frac{1}{8}}
   +
   3\cdot \frac{|G|}
   {t^{d}}\cdot \frac{3}
   {2}\cdot
   \underset{\alpha=1}
   {\overset{l}{\sum}} 
   (\frac{2}{3})^{\alpha}
   \cdot t^{\frac{1}{8}}
   +
   2\cdot
(
\frac{|G|}{t^{d}}
\cdot
\frac{3}{2}
+
\frac{|G|}{t^{d}}\cdot
t^{\frac{1}{8}})=
\end{equation*}
\begin{equation*}
     \frac{|G|}{t^{d}}\cdot 
     ((3+\frac{9}
     {2}K)t^{\frac{1}{8}}+3), 
     \text{where } 
     K=\underset{\alpha=1}
     {\overset{\infty}{\sum}} 
     (\frac{2}{3})^{\alpha}.
\end{equation*}
\indent
For $s>l$, by the proof of 
\cite[2.1]
{chudnovsky2021erdos}, 
if $k_s>2(2\frac{|G|}{(\frac{2}
{3})^s\Delta_u})^{\frac{1}{2}}$,
then there is a $(t,\frac{|G|}
{t^2})$-comb in 
$(E_{a_u},\neg E_{a_u})$
with width 
$\frac{1}{2}
(\frac{2}{3})^{s}
\Delta_u
<\frac{|G|}{t^{2d}}
<W_G$, 
contradicting the choice of 
$W_G$.
Hence, as in the proof of 
\cite[2.1]{chudnovsky2021erdos}, 
$k_s\leq 2(2\frac{|G|}
{(\frac{2}
{3})^s\Delta_u})^{\frac{1}{2}}$ and 
\begin{equation*}
    |\underset{s>l}
    {\bigcup}\text{ }C_s|
    \leq
    \frac{3^{\frac{1}{2}+1}}
    {\frac{3}{2} -(\frac{3}
    {2})^{\frac{1}
    {2}}}|G|^{\frac{1}{2}}
    ((\frac{2}{3})^{l}\Delta_u)^{\frac{1}{2}}
    \leq
    \frac{3^{\frac{1}{2}+1}}{\frac{3}{2} -(\frac{3}{2})^{\frac{1}{2}}}
    |G|^{\frac{1}{2}}(\frac{|G|}{t^{2d}})^{\frac{1}{2}}
    \leq
    \frac{3^{\frac{1}{2}+1}}{\frac{3}{2} -(\frac{3}{2})^{\frac{1}{2}}}
    \cdot
    \frac{|G|}{t^{d}}
\end{equation*}
Hence, 
\begin{equation*}
   | \{y\in \neg E(a_u,R_u): \exists x\in E(a_u,R_u) \text{ } E(x,y)\}|
   =
   |\underset{s}{\bigcup}\text{ }C_s|
   \leq
\end{equation*}
\begin{equation*}
    \frac{|G|}{t^{d}}\cdot
    ((3+\frac{9}{2}K)
    t^{\frac{1}{8}}
    +3
    +\frac{3^{\frac{1}{2}+1}}
    {\frac{3}{2} -(\frac{3}
    {2})^{\frac{1}{2}}})
    \leq
    \frac{|G|}{t^{d}}\cdot
    t^{\frac{1}{4}}.
\end{equation*}
\\
\indent
Take
\begin{equation*}
    R_{u+1}=
    \{y\in \neg 
E(a_u,R_u): \neg\exists x\in 
E(a_u,R_u) \text{ } E(x,y)\};
\end{equation*}
\indent
Let $\Delta_{u+1}$ be the 
maximal degree in 
$G[R_{u+1}]$.
Let $a_{u+1}\in R_{u+1}$ such that $|E(a_{u+1},R_{u+1})|=
\Delta_{u+1}$.  
\\\\\\
\underline{Case 2: $\Delta_u<\dfrac{|G|}{t^{2d}}$}
\\\\
By \cite[2.1]{chudnovsky2021erdos} and the choice of $W_G$,
\begin{equation*}
   |\underset{s}{\bigcup}\text{ }C_s|
    \leq
    \frac{3^{\frac{1}{2}+1}}{\frac{3}{2} -(\frac{3}{2})^{\frac{1}{2}}}|G|^{\frac{1}{2}}\Delta_u^{\frac{1}{2}}
    \leq
    \frac{3^{\frac{1}{2}+1}}{\frac{3}{2} -(\frac{3}{2})^{\frac{1}{2}}}
    |G|^{\frac{1}{2}}(\frac{|G|}{t^{2d}})^{\frac{1}{2}}
    \leq
    \frac{3^{\frac{1}{2}+1}}{\frac{3}{2} -(\frac{3}{2})^{\frac{1}{2}}}
    \cdot
    \frac{|G|}{t^{d}}  
\end{equation*}
Take
\begin{equation*}   
R_{u+1}=
\{y\in \neg E(a_u,R_u): \neg\exists x\in E(a_u,R_u) \text{ } E(x,y)\}.
\end{equation*}
\indent
Let $\Delta_{u+1}$ be the 
maximal degree in 
$G[R_{u+1}]$. Let $a_{u+1}\in 
R_{u+1}$ such that 
$|E(a_{u+1},R_{u+1})|
=
\Delta_{u+1}$.  
\\\\
\indent
This is the end of the construction of
$(a_u,\Delta_u,R_u)$.
\\\\
\indent
We use the following claim to
 show that $R_{
 \lceil
 t^{\frac{1}{8}}
 \rceil
 }$ is not too small.
\begin{claim}
    For each $u\leq \lceil
    t^{\frac{1}{8}}
    \rceil$, $|R_u|\geq
     |A|
     -
     (\frac{|G|}{t^{d}}
     +
     1
     +
      \frac{|G|}{t^{d}}\cdot
      t^{\frac{1}{4}}
      )\cdot
       u
      $.
\end{claim}
\begin{proof}
Induction on $u$:\\
\underline{$u=0$}: Since $R_0=A$, $|R_0|\geq|A|$.\\
\underline{$u+1$}: Suppose $|R_u|\geq
     |A|
     -
     (\frac{|G|}{t^{d}}
     +
     1
     +
      \frac{|G|}{t^{d}}\cdot
      t^{\frac{1}{4}}
       )\cdot
       u
      $.
      \\
      \indent
      If $\Delta_u\geq \frac{|G|}{t^{2d}}$, then
      \begin{equation*}
      R_{u+1}=
      \{y\in \neg E(a_u,R_u): \neg\exists x\in E(a_u,R_u) \text{ } E(x,y)\}. 
\end{equation*}
Since 
\begin{equation*}
    R_u
=
E(a_u;R_u)\cup
\{a_u\}\cup
\{y\in \neg E(a_u,R_u): \exists x\in E(a_u,R_u) \text{ } E(x,y)\}
\cup
R_{u+1},
\end{equation*}
we have
\begin{equation*}
    |R_u|
    \leq
    \frac{|G|}{t^{d}}
    +1
    +\frac{|G|}{t^{d}}\cdot
    t^{\frac{1}{4}}
    +
    |R_{u+1}|
    \end{equation*}
    Hence 
    \begin{equation*}
        |R_{u+1}|\geq
        |R_u|
        -
         \frac{|G|}{t^{d}}
    -1
    -\frac{|G|}{t^{d}}\cdot
    t^{\frac{1}{4}}
    .  
    \end{equation*}
    By induction,
    \begin{equation*}
      |R_{u+1}| 
      \geq
         |A|
        -
         (\frac{|G|}{t^{d}}
    +1
    +\frac{|G|}{t^{d}}\cdot
    t^{\frac{1}{4}}
    )
    \cdot (u+1).
    \end{equation*}
    If $ \Delta_u<\frac{|G|}{t^{2d}}$, then
    \begin{equation*}   
R_{u+1}=
\{y\in \neg E(a_u,R_u): \neg\exists x\in E(a_u,R_u) \text{ } E(x,y)\},
\end{equation*}
\begin{equation*}
    R_u
=
E(a_u;R_u)\cup
\{a_u\}\cup
\{y\in \neg E(a_u,R_u): \exists x\in E(a_u,R_u) \text{ } E(x,y)\}
\cup
R_{u+1}.
\end{equation*}
Hence,
\begin{equation*}
    |R_u|
    \leq
    \frac{|G|}{t^{d}}
    +1
    +
    \frac{3^{\frac{1}{2}+1}}{\frac{3}{2} -(\frac{3}{2})^{\frac{1}{2}}}
    \cdot
    \frac{|G|}{t^{d}}
    +
    |R_{u+1}|
     \leq
    \frac{|G|}{t^{d}}
    +1
    +\frac{|G|}{t^{d}}\cdot
    t^{\frac{1}{4}}
    +
    |R_{u+1}|
   .
    \end{equation*}
     By induction,
    \begin{equation*}
      |R_{u+1}| 
      \geq
         |A|
        -
         (\frac{|G|}{t^{d}}
    +1
    +\frac{|G|}{t^{d}}\cdot
    t^{\frac{1}{4}}
    )
    \cdot (u+1).
    \end{equation*}
\end{proof}
It follows that 
\begin{equation*}
R_{\lceil
t^{\frac{1}{8}}
\rceil}
\geq
\frac{|G|}{t^{d}}\cdot
(t-\lceil
t^{\frac{1}{8}}
\rceil
(1
+
\frac{t^d}{|G|}
+
t^{\frac{1}{4}}
))
\geq
\frac{|G|}{t^{d}}\cdot
(t-
2
t^{\frac{1}{8}}
(1
+
\frac{t^d}{|G|}
+
t^{\frac{1}{4}}
))
\geq
\frac{|G|}{t^{d}}\cdot t^{\frac{1}{2}}.
\end{equation*}
Depending on the size of $\Delta_{\lceil
t^{\frac{1}{8}}
\rceil
}$,
we have the following two cases.
\\
\underline{Case (i):
    $\Delta_{
    \lceil
    t^{\frac{1}{8}}
    \rceil}
    \geq
    \frac{|G|}{t^{2d}}$}.
Then $\bigl(E(a_u;R_u):
1\leq u\leq
\lceil
t^{\frac{1}{8}}
\rceil
\bigr)$ is a blockade 
such that for any $ 1\leq 
u<u'\leq 
\lceil
t^{\frac{1}{8}}
\rceil$, 
any $x\in E(a_u;R_u) $, 
$y\in E(a_{u'};R_{u'})$, 
$\neg E(x,y)$ and for all 
$ 1\leq u\leq
\lceil
t^{\frac{1}{8}}
\rceil$,
    $|E(a_u;R_u)|\geq
     \frac{|G|}{t^{2d}}$. Since $G$ is $\tau$-critical, $G$ has an induced cograph of size $\geq 
     (\frac{|G|}{t^{2d}})^{\tau}
     \cdot 
     \lceil
     t^{\frac{1}{8}}
     \rceil
     \geq
     |G|^{\tau}
     \cdot t^{\frac{1}{8}-2d\tau}
     >|G|^{\tau}$,
     a contradiction.
    \\\\
  \underline{Case (ii):
    $\Delta_{
    \lceil
    t^{\frac{1}{8}}
    \rceil
    }
    <
    \frac{|G|}{t^{2d}}$}. In this case, we follow the proof of \cite[3.1]{chudnovsky2021erdos}:\\
    Let
    \begin{equation*}
        W_{u}
        =
        \{y\in \neg E(a_u,R_u): \exists x\in E(a_u,R_u) \text{ } E(x,y)\}.
    \end{equation*}
    Let $u_0$ be the largest such that $\Delta_{u_0}>0$. \\Let $R=
    R_{u_0}\setminus
    (E(a_{u_0},R_{u_0})\cup
    \{a_{u_0}\}\cup
    W_{u_0})$.
Then
\begin{equation*}
    R_{
    \lceil
    t^{\frac{1}{8}}
    \rceil
    }
    =
    \underset{u\geq
    \lceil
    t^{\frac{1}{8}}
    \rceil
    }{\bigcup}
    E(a_u,R_u)
    \cup
     \underset{u\geq
     \lceil
    t^{\frac{1}{8}}
    \rceil
    }{\bigcup}
    \text{ } W_u
    \cup
    \{a_u:u\geq 
    \lceil
    t^{\frac{1}{8}}
    \rceil
    \}
    \cup
    R
\end{equation*}
For $u\geq
\lceil
t^{\frac{1}{8}}
\rceil
$, let $x_u:=\dfrac{|E(a_u,R_u)|}
{|R_{
\lceil
t^{\frac{1}{8}}
\rceil
}|}$ and let $H_u\subseteq E(a_u,R_u)$ be an induced cograph of maximal size.
Since $\Delta_{
\lceil
t^{\frac{1}{8}}
\rceil
}
    <
    \frac{|G|}{t^{2d}}$, for $u\geq 
    \lceil
    t^{\frac{1}{8}}
    \rceil
    $,
    \begin{equation*}
        x_u
        \leq
        \frac{\frac{|G|}{t^{2d}}}{\frac{|G|}
        {t^{d}}\cdot t^{\frac{1}{2}}}
        =
        \frac{t^{d-\frac{1}{2}}}{t^{2d}}
        =
        t^{-d-\frac{1}{2}}.
    \end{equation*}
Since $G$ is $\tau$-critical, for each $u\geq 
\lceil
t^{\frac{1}{8}}
\rceil
$,
\begin{equation*}
    |H_u|
    \geq
    |E(a_u,R_u)|^\tau
    =
    x_u^{\tau}
    \cdot 
    |R_{
    \lceil
    t^{\frac{1}{8}}
    \rceil
    }|^{\tau}
    \geq
     x_u^{\tau}
    \cdot 
    (\frac{|G|}{t^{d}}\cdot t^{\frac{1}{2}})
    ^{\tau}
    =
    x_u^{\tau}
    \cdot 
     (\frac{1}{t^{d-\frac{1}{2}}})
    ^{\tau}
    \cdot
    |G|^{\tau}.
\end{equation*}
Since $\underset{u\geq 
\lceil
t^{\frac{1}{8}}
\rceil
}
{\bigcup}$ $H_u$ is a cograph and $G$ is $\tau$-critical, we have
\begin{equation*}
    |\underset{u\geq
    \lceil
    t^{\frac{1}
    {8}}
    \rceil
    }{\bigcup}\text{ } H_u|
    =
    \underset{u\geq 
    \lceil
    t^{\frac{1}
    {8}}
    \rceil
    }{\sum}\text{ } |H_u|
    <
    |G|^{\tau}. 
\end{equation*}
Hence,
\begin{equation*}
    \underset{u\geq 
    \lceil
    t^{\frac{1}
    {8}}
    \rceil
    }{\sum}\text{ }  x_u^{\tau}
    \cdot 
     (\frac{1}{t^{d-\frac{1}{2}}})
    ^{\tau}
    \leq
     \underset{u\geq 
     \lceil
     t^{\frac{1}
    {8}}
    \rceil
    }{\sum}\text{ } \frac{|H_u|}{|G|^{\tau}}
    =
    \frac{1}{|G|^{\tau}}
    \cdot
    \underset{u\geq 
    \lceil
    t^{\frac{1}
    {8}}
    \rceil
    }{\sum}\text{ } |H_u|
    <1
\end{equation*}
and
\begin{equation*}
     \underset{u\geq 
     \lceil
     t^{\frac{1}
    {8}}
    \rceil
    }{\sum}\text{ }  x_u^{\tau}
    <
    (\frac{1}{t^{d-\frac{1}{2}}})
    ^{-\tau}
\end{equation*}
Since $\{a_u:u\geq 
\lceil
t^{\frac{1}{8}}
\rceil
\}$ does not have an edge, it is a cograph and 
\begin{equation*}
    |\{a_u:u\geq 
    \lceil
    t^{\frac{1}{8}
    }
    \rceil
    \}|
    <
    |G|^\tau.
\end{equation*}
So
\begin{equation*}
    \frac{ |\{a_u:u\geq 
    \lceil
    t^{\frac{1}{8}}
    \rceil
    \}|}
    {|R_{
    \lceil
    t^{\frac{1}{8}}
    \rceil
    }|}
    \leq
       \frac{|\{a_u:u\geq 
       \lceil
       t^{\frac{1}{8}}
       \rceil
       \}|}
    {\frac{|G|}{t^{d}}\cdot t^{\frac{1}{2}}}
    \leq\frac{|G|^\tau}
       {\frac{|G|}{t^{d}}\cdot t^{\frac{1}{2}}}
       =
       |G|^{\tau-1}\cdot
       t^{d-\frac{1}{2}}
\end{equation*}
We estimate $\frac{| \underset{u\geq
    \lceil
    t^{\frac{1}{8}}
    \rceil
    }{\bigcup}
    E(a_u,R_u)|}
    {|R_{
    \lceil
    t^{\frac{1}{8}}
    \rceil
    }|}$:
\begin{equation*}
    \frac{| \underset{u\geq
    \lceil
    t^{\frac{1}{8}}
    \rceil
    }{\bigcup}
    E(a_u,R_u)|}{|R_{
    \lceil
    t^{\frac{1}{8}}
    \rceil
    }|}
    =
    \frac{ \underset{u\geq 
    \lceil
    t^{\frac{1}
    {8}}
    \rceil
    }{\sum} \text{ }
    |E(a_u,R_u)|}
    {|R_{
    \lceil
    t^{\frac{1}{8}}
    \rceil
    }|}
    =
    \underset{u\geq 
    \lceil
    t^{\frac{1}
    {8}}
    \rceil
    }{\sum}\text{ }
    \frac{|E(a_u,R_u)|}
    {|R_{
    \lceil
    t^{\frac{1}{8}}
    \rceil
    }|}
    =
    \underset{u\geq 
    \lceil
    t^{\frac{1}
    {8}}
    \rceil
    }{\sum} \text{ } x_u
    =
\end{equation*}
\begin{equation*}
    \underset{u\geq 
    \lceil
    t^{\frac{1}
    {8}}
    \rceil
    }{\sum} \text{ } x_u^{\tau}
    \cdot
    x_u^{1-\tau}
    \leq
    \underset{u\geq 
    \lceil
    t^{\frac{1}
    {8}}
    \rceil
    }{\sum} \text{ } x_u^{\tau}
    \cdot
     t^{(-d-\frac{1}{2})(1-\tau)} 
\end{equation*}
\begin{equation*} 
\leq
      (\frac{1}{t^{d-\frac{1}{2}}})
    ^{-\tau}
      \cdot
     t^{-d-\frac{1}{2}+(d+\frac{1}{2})\tau}
     =
         t^{-d-\frac{1}{2}+2d\tau}
\end{equation*}
We bound $\frac{|   \underset{u\geq
    \lceil
    t^{\frac{1}{8}}
    \rceil
    }{\bigcup}
    \text{ } W_u|}
    {|R_{
    \lceil
    t^{\frac{1}{8}}
    \rceil
    }|}$:
\begin{equation*}
    \frac{|   \underset{u\geq
    \lceil
    t^{\frac{1}{8}}
    \rceil
    }{\bigcup}
    \text{ } W_u|}
    {|R_{
    \lceil
    t^{\frac{1}{8}}
    \rceil
    }|}
    =
    \frac{ \underset{u\geq
    \lceil
    t^{\frac{1}{8}}
    \rceil
    }{\sum}
    \text{ } |W_u|}
    {|R_{
    \lceil
    t^{\frac{1}{8}}
    \rceil
    }|}
    =
      \underset{u\geq
    \lceil
    t^{\frac{1}{8}}
    \rceil
    }{\sum}\text{ }
    \frac{ |W_u|}
    {|R_{
    \lceil
    t^{\frac{1}{8}}
    \rceil
    }|}\leq
\end{equation*}
\begin{equation*}
         \underset{u\geq
    \lceil
    t^{\frac{1}{8}}
    \rceil
    }{\sum}  \text{ }
     \frac{3^{\frac{1}{2}+1}}{\frac{3}{2} -(\frac{3}{2})^{\frac{1}{2}}}(\frac{|G|}{|R_{
     \lceil
     t^{\frac{1}{8}}
     \rceil
     }|})^{\frac{1}{2}}
     (\frac{\Delta_u}{|R_{
     \lceil
     t^{\frac{1}{8}}
     \rceil
     }|})^{\frac{1}{2}}\leq
\end{equation*}
\begin{equation*}
     \frac{3^{\frac{1}{2}+1}}{\frac{3}{2} -(\frac{3}{2})^{\frac{1}{2}}}
     \cdot
     (\frac{|G|}{\frac{|G|}{t^{d}}\cdot t^{\frac{1}{2}}})^{\frac{1}{2}}
     \underset{u\geq
    \lceil
    t^{\frac{1}{8}}
    \rceil
    }{\sum}
    \text{ } x_u^{\frac{1}{2}}=
\end{equation*}
\begin{equation*}
     \frac{3^{\frac{1}{2}+1}}{\frac{3}{2} -(\frac{3}{2})^{\frac{1}{2}}}
     \cdot
     t^{\frac{d}{2}-\frac{1}{4}}
     \underset{u\geq
    \lceil
    t^{\frac{1}{8}}
    \rceil
    }{\sum}
    \text{ } x_u^{\tau} x_u^{\frac{1}{2}-\tau}\leq
\end{equation*}
\begin{equation*}
     \frac{3^{\frac{1}{2}+1}}{\frac{3}{2} -(\frac{3}{2})^{\frac{1}{2}}}
     \cdot
     t^{\frac{d}{2}-\frac{1}{4}}
     \underset{u\geq
    \lceil
    t^{\frac{1}{8}}
    \rceil
    }{\sum}
    \text{ } x_u^{\tau}
    \cdot
    (t^{-d-\frac{1}{2}})^{\frac{1}{2}-\tau}\leq
\end{equation*}
\begin{equation*}
      \frac{3^{\frac{1}{2}+1}}{\frac{3}{2} -(\frac{3}{2})^{\frac{1}{2}}}
     \cdot
     t^{-\frac{1}{2}+
     (d+\frac{1}{2})
     \tau}
     \cdot
      (\frac{1}{t^{d-\frac{1}{2}}})
    ^{-\tau}=
\end{equation*}
\begin{equation*}
      \frac{3^{\frac{1}{2}+1}}{\frac{3}{2} -(\frac{3}{2})^{\frac{1}{2}}}
    \cdot
     t^{-\frac{1}{2}+
     2d\tau}.
\end{equation*}
Since $u_0$ is the largest such that $\Delta_{u_0}>0$, $R$ has no edge. Hence, $|R|<|G|^{\tau}$.
It follows that 
\begin{equation*}
    1=
    \frac{|R_{
    \lceil
    t^{\frac{1}{8}}
    \rceil
    }|}{|R_{
    \lceil
    t^{\frac{1}{8}}
    \rceil
    }|}
    =
   \frac{|\underset{u\geq
    \lceil
    t^{\frac{1}{8}}
    \rceil
    }{\bigcup}
    E(a_u,R_u)
    \cup
     \underset{u\geq
    \lceil
    t^{\frac{1}{8}}
    \rceil
    }{\bigcup}
    \text{ } W_u
    \cup
    \{a_u:u\geq \lceil
    t^{\frac{1}{8}}
    \rceil
    \}
    \cup
    R|}
    {|R_{
    \lceil
    t^{\frac{1}{8}}
    \rceil
    }|}
    \leq
\end{equation*}
\begin{equation*}
     t^{-d-\frac{1}{2}+2d\tau}
     +
     \frac{3^{\frac{1}{2}+1}}{\frac{3}{2} -(\frac{3}{2})^{\frac{1}{2}}}
     \cdot
     t^{-\frac{1}{2}+
     2d\tau}
     +
      |G|^{\tau-1}
      \cdot
       t^{d-\frac{1}{2}}
       +
       |G|^{\tau-1}
       \cdot
       t^{d-\frac{1}{2}}
       \leq 
\end{equation*}
\begin{equation*}
     t^{-d-\frac{1}{2}+2d\tau}
     +
     \frac{3^{\frac{1}{2}+1}}{\frac{3}{2} -(\frac{3}{2})^{\frac{1}{2}}}
    \cdot
     t^{-\frac{1}{2}+
     2d\tau}
     +
      |G|
      ^{\tau-\frac{1}{2d}}
       \cdot
       2^{d+\frac{1}{2}}
       \leq
\end{equation*}
\begin{equation*}
         t^{-d-\frac{1}{2}+2d\tau}
     +
     \frac{3^{\frac{1}{2}+1}}{\frac{3}{2} -(\frac{3}{2})^{\frac{1}{2}}}
    \cdot
     t^{-\frac{1}{2}+
     2d\tau}
     +
    |G|
      ^{-(d+1)\tau}
       \cdot 
       2^{d+\frac{1}{2}}
\end{equation*}
which is $<1$ by the choice 
of $\tau$ and $t$ and by the 
fact that $|G|^{\tau}\geq 2$ 
since a graph with two 
vertices is always a cograph. 
This gives a contradiction.
Hence, there exists a comb as 
described in the statement.
\end{proof}

\bibliographystyle{alpha}
\bibliography{Ref}
\end{document}